\documentclass[12pt]{article}        

 \usepackage[T1]{fontenc}             
 \usepackage[width=16cm, height=22cm]{geometry}   

 \usepackage{mathtools} 
 \usepackage{amssymb}          
\usepackage{hyperref} 
 \usepackage[amsmath,thmmarks,hyperref]{ntheorem}

 \usepackage{icomma}                   
 \usepackage{enumerate}          
 \usepackage{color}                 
 \usepackage{setspace}                 
 \usepackage{needspace}               
 \usepackage{url}                    
 \usepackage{mathrsfs}                
 \usepackage{esint}    
 \usepackage{slashed}

 \numberwithin{equation}{section}
 
 \usepackage[numbers,square]{natbib}
 \usepackage{tikz}
\usetikzlibrary{arrows.meta}
\usetikzlibrary{cd}

\theoremstyle{nonumberplain}  
\theoremheaderfont{\itshape}  
\theorembodyfont{\normalfont}  
\theoremseparator{.}  
\theoremsymbol{$\Box$}
\newtheorem{proof}{Proof} 
  
\theoremstyle{plain}  
\theoremheaderfont{\normalfont\sc}  
\theorembodyfont{\normalfont}  
\theoremsymbol{~}
\theoremseparator{.}  
\newtheorem{proposition}{Proposition}[section]

\newtheorem{theorem}[proposition]{Theorem}   

\theorembodyfont{\normalfont}  
 
\newtheorem{remark}[proposition]{Remark}

\newtheorem{definition}[proposition]{Definition} 
\newtheorem{notation}[proposition]{Notation} 

\newcommand{\R}{\mathbb{R}}

\newcommand{\Z}{\mathbb{Z}}
\newcommand{\C}{\mathbb{C}}
\newcommand{\dd}{\mathrm{d}}
\newcommand{\tr}{\mathrm{tr}}

\newcommand{\ch}{\mathrm{ch}}

\newcommand{\cc}{\mathbf{c}}
\DeclareMathOperator{\str}{\mathrm{str}}

\renewcommand{\tilde}{\widetilde}

\newcommand{\T}{\mathbb{T}}

\renewcommand{\L}{\text{\normalfont\sffamily L}}

\newcommand{\pf}{\mathrm{pf}}
\newcommand{\Pf}{\text{\normalfont\sffamily Pf}}

\newcommand{\sgn}{\mathrm{sgn}}

\newcommand{\Cl}{\mathrm{Cl} }

\newcommand{\B}{\mathsf{B}}

\DeclareMathOperator{\Str}{\mathrm{Str}}

\title{Construction of the Supersymmetric Path Integral: \\A Survey}
\author{Matthias Ludewig\footnote{The University of Adelaide. matthias.ludewig@adelaide.edu.au}}

\begin{document}

\maketitle

\begin{center}\parbox{12cm}{
{\footnotesize This is a survey based on the joint work \cite{GueneysuLudewig,  HanischLudewig1} with Florian Hanisch and Batu G\"uneysu reporting on a rigorous construction of the supersymmetric path integral associated to compact spin manifolds.}
}\end{center}

\section{Introduction}

A way to understand the geometry of the loop space\footnote{Throughout, we denote $\T = S^1 = \R/\Z$; the smooth loop space is then defined by $\L X = C^\infty(\T, X)$.} $\L X$ of a manifold $X$ is by studying its differential forms. On finite-dimensional manifolds, a key feature of differential forms is that they can be {\em integrated}, which gives a linear functional on the space of differential forms\footnote{We remark that even in the finite-dimensional case, if the manifold is non-compact, then the integration functional will only be defined on a suitable subset of {\em integrable forms}. The same will be true in the infinite-dimensional case.}. Of course, one of the fundamental properties of this integration functional is that it is only non-zero on forms of highest degree; at first glance, this seems to make the task of defining such an integration functional in the infinite-dimensional context of the loop space impossible, as there is no top degree. 

However, if we fix a Riemannian metric on $X$ and define the {\em canonical two-form} on $\L X$ by setting
\begin{equation} \label{CanonicalTwoForm}
  \omega[v, w] := \int_{\T} \bigl\langle v(t), \nabla_{\dot{\gamma}} w(t)\bigr\rangle\, \dd t
\end{equation}
for $\gamma \in \L X$, $v, w \in T_\gamma \L X = C^\infty(\T, \gamma^*TX)$, it turns out that there is a natural way to make sense of the top-degree component of the wedge product $e^\omega \wedge \theta$ for suitable forms $\theta$, by using simple analogies to the finite-dimensional situation (here $e^\omega$ denotes the exponential of $\omega$ in the algebra $\Omega(\L X)$ of differential forms). This {\em top degree component} $[e^\omega \wedge \theta]_{\mathrm{top}}$ should be seen as the pairing of $e^\omega \wedge \theta$ with the volume form corresponding to the $L^2$-metric on $\L X$ (which must remain heuristic as there is no such volume form). Using this notion, an integration functional can then be defined using the Wiener measure.

\medskip

\noindent \textbf{Relation to the Atiyah-Singer index theorem.} There is another side to this story of integrating differential forms on the loop space, which is our main motivation.
More than 30 years ago, it was observed by Atiyah and Witten \cite{MR816738} that there is a very short and conceptual, but formal, i.e.\ non-rigorous, proof of the Atiyah-Singer index theorem using a supersymmetric version of the Feynman path integral. In physics terms, this is the path integral of the $\mathcal{N}=1/2$ supersymmetric $\sigma$-model \cite{AlvarezGaume}. Reformulating the supergeometry appearing in the work of Alvarez-Gaum\'e in the language of differential forms, Atiyah was led to consider the differential form integral
\begin{equation} \label{FormalDefI}
I[\theta] ~ \stackrel{\text{formally}}{=} ~ \int_{\mathsf{L} X} e^{-S + \omega} \wedge \theta
\end{equation}
over the loop space of a Riemannian (spin) manifold $X$, for suitable differential forms $\theta \in \Omega(\mathsf{L} X)$, where
\begin{equation} \label{EnergyAndTwoForm}
  S(\gamma) = \frac{1}{2} \int_{\T} |\dot{\gamma}(t)|^2 \mathrm{d} t
\end{equation}
is the usual energy functional, and $\omega$ is the canonical two-form defined in \eqref{CanonicalTwoForm}. Atiyah proceeds with a series of formal manipulations 
allowing him to rewrite \eqref{FormalDefI} as a Wiener integral. Then, using the Feynman-Kac formula, he identifies this Wiener integral with the supertrace of the heat semigroup associated to the Dirac operator and thus (via the McKean-Singer formula) with the index of the Dirac operator.

On the other hand, the loop space has a natural $\T$-action by rotation of loops, and the differential form $S- \omega$ is closed with respect to the {\em equivariant differential} 
\begin{equation} \label{EquivariantDifferential}
d_K := d + \iota_K,
\end{equation}
where $\iota_K$ denotes insertion of the generating vector field $K(\gamma)= \dot{\gamma}$  of the rotation action. Hence, if the given $\theta$ also satisfies $d_K \theta = 0$, then the composite differential form $e^{-S+\omega} \wedge \theta$ considered above is equivariantly closed as well. Motivated by this observation, Atiyah  formally\footnote{Meaning that one pretends that $\L X$ is finite-dimensional.} applies a Duistermaat-Heckmann type formula \cite{BerlineVergne, DuistermaatHeckmann}, in order to localize the integral to the fixed point set with respect to the rotation action, which is precisely the set of constant loops. Now there is an obvious inclusion map  $i: X \rightarrow \mathsf{L} X$ identifying $X$ with this fixed point set, and one has the {\em localization formula}
\begin{equation} \label{LocalizationFormula}
  I[\theta] ~ \stackrel{\text{formally}}{=} ~ \int_X \widehat{A}(X) \wedge i^*\theta.
\end{equation}
It was later observed by Bismut \cite{MR786574} that this can be used to (formally) prove the twisted Atiyah-Singer index theorem, by considering special differential forms on $\mathsf{L} X$ defined from the data of a vector bundle with connection on $X$, which today are called {\em Bismut-Chern characters}.

\medskip

\noindent \textbf{Our work.} In this survey, we give an account of a recent project \cite{HanischLudewig1, GueneysuLudewig} that carries out a rigorous construction of the supersymmetric path integral map $I$ described above. The map should have the following properties.
\begin{enumerate}
\item[(i)] The map $I$ is defined on some large subset of $\Omega(\L X)$ of {\em integrable forms}, which at least includes the Bismut-Chern characters defined by Bismut \cite{MR786574}.
\item[(ii)] For any integrable differential form $\theta$ with $d_K\theta = 0$, the map $I$ satisfies the localization formula \eqref{LocalizationFormula}.
\end{enumerate}
We remark that in particular, (ii) implies that $I$ is coclosed with respect to $d_K$; in physics language, this means that the path integral is {\em supersymmetric}, where the idea is that the functional is invariant under the odd symmetry generated by $d_K$. Of course, the properties (i)-(ii) do not fix $I$ uniquely, since e.g.\ the functional $I_0(\theta)$, just defined as the right hand side of \eqref{LocalizationFormula} satisfies both requirements tautologically. To obtain a reasonable problem, we therefore add the following rather heuristic requirement.
\begin{enumerate}
\item[(iii)] The map $I$ is given by formula \eqref{FormalDefI} in a suitable sense.
\end{enumerate}
In our work, we construct such a map $I$. Notice that property (ii) follows if $I$ is homologous to the map $I_0$ defined by the right hand side of \eqref{LocalizationFormula}; however, we emphasise that our construction is {\em geometric}. In other words, we construct $I$ as a {\em cochain} rather than an equivalence class in cohomology. 

In fact, we provide two different constructions of the map $I$: In \cite{HanischLudewig1}, a stochastic approach is taken to construct $I$ starting from property (iii); it is not necessarily apparent from this approach, however, that the map constructed that way has property (ii). This is fixed in \cite{GueneysuLudewig}, where we use methods from non-commutative geometry to define a map which $-$ using a fancy version of Getzler rescaling $-$ can be shown to satisfy (ii). The equivalence of these constructions is then established in \cite{HanischLudewig1}.

\medskip

In this survey, we proceed by highlighting the first construction, as described in \cite{HanischLudewig1}; afterwards, we discuss the second construction, as given in \cite{GueneysuLudewig}. In the final section, we connect the two approaches and discuss the localization formula \eqref{LocalizationFormula} and its application to the Bismut-Chern characters.

\section{First Construction: The Top Degree Functional} \label{SectionFirstConstruction}

In this section, we give a quick overview of the construction of the path integral map $I$ portrayed in the introduction, following \cite{HanischLudewig1}. 

\medskip

\noindent \textbf{The Wiener measure.} The construction is essentially based on the Wiener measure $\mathbb{W}$, a certain measure on the continuous loop space\footnote{In fact, the Wiener measure is defined on any space of paths, but here we restrict to the loop space.} $\L_c X = C(\T, X)$ of a Riemannian manifold $X$. The Wiener integral of so-called {\em cylinder functions} is easy to describe. These are functions $F: \L_c X \rightarrow \C$ of the form
\begin{equation} \label{CylinderFunction}
F(\gamma) = f\bigl(\gamma(\tau_1), \dots, \gamma(\tau_N)\bigr)
\end{equation}
for some  $f \in C(M^N)$ and  $0 \leq \tau_1 < \dots < \tau_N < 1$; the formula for their Wiener integral $\mathbb{W}[F]$ is\footnote{In the formula, we adopt the notation $\tau_{N+1} :=1 + \tau_1$ and $x_{N+1} := x_1$.}
\begin{equation} \label{FormulaWienerMeasure}
  \mathbb{W}[F] \stackrel{\text{def}}{=} \int_X \cdots \int_X f(x_1, \dots, x_N) \left(\prod_{j=1}^N p_{\tau_{j+1}-\tau_{j}}(x_{j}, x_{j+1})\right) \dd x_1 \cdots \dd x_N, 
\end{equation}
where $p_t(x, y)$ is the {\em heat kernel} of $X$, i.e.\ the fundamental solution to the heat equation. By the extension theorem and the continuity theorem of Kolmogorov, this determines $\mathbb{W}[F]$ uniquely for all bounded functions $F$ on $\L_c X$. 

For $X = \R^n$, one has the explicit formula
\begin{equation*}
  p_t(x, y) = (2 \pi t)^{-n/2} \exp \left( - \frac{|x-y|^2}{2t} \right)
\end{equation*}
for the heat kernel. After inserting this into \eqref{FormulaWienerMeasure} for $F$ a cylinder function, some elementary manipulations give the result
\begin{equation} \label{FiniteDimApprox}
   \mathbb{W}[F] = \left(\prod_{j=1}^N {\bigl(2 \pi (\tau_j-\tau_{j-1})\bigr)^{-n/2}}\right)\int_{\R^n} \cdots \int_{\R^n} F(\gamma) e^{-S(\gamma)} \dd x,
\end{equation}
where $\gamma = \gamma_x$ is the piecewise linear loop with $\gamma(\tau_j) = x_j$ and as usual, $S$ is the energy functional \eqref{EnergyAndTwoForm}. This formula has an extension to manifolds \cite{MR1698956, MR2482215,  MR3663619}. In fact, formulas like \eqref{FiniteDimApprox} go all the way back to Feynman \cite{FeynmanHibbs}, constituting the starting point for his path integral approach to quantum mechanics. Taking the limit over $N$, \eqref{FiniteDimApprox} leads to the heuristic formula
\begin{equation} \label{HeuristicWiener}
  \mathbb{W}[F] \stackrel{\text{formally}}{=} \frac{1}{C}\int_{\L_c X} F(\gamma) e^{-S(\gamma)} \dd \gamma
\end{equation}
for a suitable constant $C$; in other words, the slogan is that the Wiener measure has the density function $e^{-S}$ with respect to the ``Riemannian volume measure'' $\dd \gamma$ on the loop space $\L_c X$. Of course, there are several well-known problems with this formula that make it remain heuristic, first and foremost the non-existence of the measure $\dd \gamma$ and the infinitude of the constant $C$.

\medskip

\noindent \textbf{Formal definition of the path integral map.}
Ignoring the difference between the smooth and the continuous loop space for the moment, we record that we do not know yet how to integrate differential forms, but at least the Wiener measure enables us to integrate {\em functions} over the loop space. However, if $M$ is an {\em oriented} (for now finite-dimensional) Riemannian manifold, integrating differential forms and functions is essentially the same thing: The two are related by the formula
\begin{equation} \label{RelationBetweenIntegrals}
  \int_M \theta = \int_M [\theta_y]_{\mathrm{top}} \dd y
\end{equation}
for differential forms $\theta$, where the left hand side is to be understood as a differential form integral (determined by the orientation) and the right hand side is the integration map for functions determined by the Riemannian structure. While the latter integration map does not depend on the choice of orientation, the integrand
\begin{equation*}
  [\theta_y]_{\mathrm{top}} \stackrel{\text{def}}{=} \langle \theta_y, \mathrm{vol}_y\rangle,
\end{equation*}
a function on $M$ called the {\em top degree component} of $\theta$, does, as the sign of the volume form $\mathrm{vol}$ depends on the orientation. In supergeometry, this functional is often called {\em Grassmann} or {\em Berezin integral}, after \cite{Berezin}.

The idea is now to apply the observation above to the heuristic formula \eqref{FormalDefI} for the path integral map. Starting from this formula, we obtain the chain of identifications
\begin{equation} \label{RewriteI}
  \int_{\L X} e^{-S + \omega} \wedge \theta \stackrel{\text{formally}}{=} \int_{\L X} [e^\omega \wedge \theta_\gamma]_{\mathrm{top}} \,e^{-S(\gamma)} \dd \gamma \stackrel{\text{formally}}{=} \mathbb{W}\bigl[ [e^\omega \wedge \theta]_{\mathrm{top}}\bigr];
\end{equation}
here in the first step, we formally applied \eqref{RelationBetweenIntegrals} to this infinite-dimensional example, while in the second step, we recognized the right hand side of the heuristic formula \eqref{HeuristicWiener}, for the integrand $F(\gamma) = [e^\omega \wedge \theta_\gamma]_{\mathrm{top}}$.

With a view on the right hand side of \eqref{RewriteI}, the non-trivial task that remains is to provide meaning for the top degree component $[e^\omega \wedge \theta]_{\mathrm{top}}$ of the differential form $e^\omega \wedge \theta$ as a $\mathbb{W}$-integrable function on $\L_c X$; this is the main achievement of the paper \cite{HanischLudewig1}; we outline the construction below. 

\begin{remark}
The formal manipulations conducted in \eqref{RewriteI} are more or less well-known. However, in the literature, the differential form $\theta$ is either constant equal to one (see \cite{MR816738}) or taken to be a Bismut-Chern character (see \cite{MR786574}). In both cases, the top degree component $[e^\omega \wedge \theta]_{\mathrm{top}}$ can be defined (and computed) using {\em ad hoc} methods. The novelty of the approach taken in our paper \cite{HanischLudewig1} is that we allow a very general class of differential forms $\theta$ to be plugged into our top degree functional, in order to obtain a general definition of the path integral.
\end{remark}

\medskip

\noindent \textbf{Definition of the top degree functional.}  To explain the definition of our top degree functional, notice that the canonical two-form $\omega$ has the form $\omega[v, w] = \langle v, A w\rangle_{L^2}$ in terms of the $L^2$ scalar product, where $A = \nabla_{\dot{\gamma}}$, a skew-adjoint operator on $C^\infty(\T, \gamma^* TX)$. Now if $V$ is an arbitrary finite-dimensional, oriented Euclidean vector space and $\omega \in \Lambda^2 V^\prime$ has the form $\omega[v, w] = \langle v, A w\rangle_{V}$ for an invertible skew-adjoint operator $A$ on $V$, one has the result
\begin{equation} \label{FinDimAnalogy}
 [e^\omega \wedge \vartheta_1 \wedge \cdots \wedge \vartheta_N]_{\mathrm{top}} = \pf(A) \,\pf \Bigl( \langle \vartheta_a, A^{-1} \vartheta_b\rangle_V \Bigr)_{1 \leq a, b \leq N},
\end{equation}
for $\vartheta_1, \dots, \vartheta_N \in V^\prime$, where $\pf$ stands for the {\em Pfaffian} of a skew-symmetric matrix, c.f.\ \cite[Prop.~1]{Lott}.
In the case that $A$ is not invertible, there is a similar, slightly more complicated formula, for details, c.f.\ \cite{HanischLudewig1}. This allows to define the top degree functional on the infinite-dimensional Euclidean vector space $V = C^\infty(\T, \gamma^*TX)$ by analogy: In case that $A = \nabla_{\dot{\gamma}}$ is invertible, we can set
\begin{equation} \label{DefTopDegree}
  q(\theta_1 \wedge \cdots \wedge \theta_N) \stackrel{\mathrm{def}}{=} \pf_\zeta(\nabla_{\dot{\gamma}}) \pf\Bigl( \langle \theta_a, \nabla_{\dot{\gamma}}^{-1} \theta_b\rangle_V \Bigr)_{1 \leq a, b \leq N}
\end{equation}
for $\theta_1, \dots, \theta_N \in C^\infty(\T, \gamma^*T^\prime X)$ and if $\nabla_{\dot{\gamma}}$ is not invertible, it is invertible on the orthogonal complement of its (always finite-dimensional) kernel, which allows to employ the generalization of the formula \eqref{FinDimAnalogy} mentioned above. Hence heuristically, $q(\theta)$ is the ``top degree component'' of the differential form $e^\omega \wedge \theta$. 

In \eqref{DefTopDegree}, $\pf_\zeta(\nabla_{\dot{\gamma}})$ denotes the zeta-regularized Pfaffian of $\nabla_{\dot{\gamma}}$, a square root of its zeta-regularized determinant. This quantity is not a number but rather an element of the {\em Pfaffian line} $\Pf_\gamma$, a certain one-dimensional real vector space canonically associated to $\gamma$; this reflects the fact that there is no na\"ive concept of orientation on the infinite-dimensional vector space $V$. 
These Pfaffian lines glue together to the so-called {\em  Pfaffian line bundle} $\Pf$ on $\L X$, which is related to the spin condition: By the work of Stolz-Teichner and Waldorf \cite{StolzTeichner, MR3493404}, a spin structure  on $X$ gives an orientation of the loop space $\L X$, in the sense that it provides a canonical trivialization of the Pfaffian line bundle and turns the top degree component \eqref{DefTopDegree} into an honest number. This is the reason why the spin condition is important to define our path integral.

\begin{remark} 
This is analogous to the fact that on a finite-dimensional {\em non-oriented} manifold, the top degree component is also a section of a real line bundle, the {\em orientation bundle}, which is trivialized by an orientation.
\end{remark}

The main result of our paper \cite{HanischLudewig1} is then the following formula, which provides a way to actually compute its value and the value of the integral map.

\begin{theorem} \label{ThmMainIdentification}
Suppose that $X$ is a spin manifold with spinor bundle $\Sigma$. Then the top-degree component defined above is canonically a number, and it is given by the formula
\begin{equation} \label{FormulaFromThm}
    q(\theta_N \wedge \cdots \wedge \theta_1) = 2^{-N/2}  \sum_{\sigma \in S_N} \mathrm{sgn}(\sigma) \int_{\Delta_N} \str \left([\gamma\|_{\tau_N}^1]^\Sigma \prod_{a=1}^N \cc\bigl(\theta_{\sigma_a}(\tau_a)\bigr) [\gamma\|_{\tau_{a-1}}^{\tau_a}]^\Sigma\right)\dd \tau.
\end{equation}
Here $\Delta_N = \{ 0 \leq \tau_1 \leq \dots \leq \tau_N \leq 1 \}$ is the standard simplex, $[\gamma\|_{\tau_{a-1}}^{\tau_a}]^\Sigma$ denotes parallel translation in the spinor bundle along the loop $\gamma$ and $\cc$ denotes Clifford multiplication. Moreover, $\str$ is the supertrace of the spinor bundle\footnote{Throughout, we take the {\em real spinor bundle}, a bundle of irreducible $\Cl(TX)$-$\Cl_n$-bimodules. In any dimension, the space $\mathrm{End}_{\Cl_n}(\Sigma)$ of endomorphisms of $\Sigma$ commuting with the right $\Cl_n$-action carries a canonical supertrace.}.
\end{theorem}



\noindent \textbf{Rigorous definition of the path integral map.} At this point, the top degree map assigns to a certain class of differential forms $\theta$ on the loop space $\L X$ of a spin manifold $X$ the smooth function $q(\theta)$ on $\L X$, to be interpreted as the ``top degree component'' of $e^\omega \wedge \theta$. The problem now is that we need $q(\theta)$ to be a function on the {\em continuous} loop space $\L_c X$, in order to be able to integrate with respect to the Wiener measure, as in \eqref{RewriteI}. 

One problem here when looking at formula \eqref{FormulaFromThm} is that a loop has to be absolutely continuous in order to define the parallel transport along it. A solution to this problem is provided by the notion of {\em stochastic parallel transport}: As ultimately, $I$ is defined by Wiener integration, it suffices to have $q(\theta)$ defined as a measurable function only (with respect to the Wiener measure $\mathbb{W}$). This is achieved by interpreting the occurrences of the parallel transport in  \eqref{FormulaFromThm} in the stochastic sense, which provides a stochastic extension $\tilde{q}$ of the function $q$; for details on the stochastic parallel transport, see e.g.\  \cite{Emery, GueneysuVectorValuedFK, HackThal}.

To discuss the possible integrands $\theta$, notice that since $\L_c X \subset \L X$ is dense, we can consider $\Omega(\L_c X)$ as a subspace of $\Omega(\L X)$. 

\begin{notation} \label{NotationD}
Denote my $\mathscr{D} \subseteq \Omega(\L_c X) \subset \Omega(\L X)$ the space of differential forms $\theta$ that are wedge products of one forms that are uniformly bounded.
\end{notation}

For elements $\theta \in \mathscr{D}$, the function $\tilde{q}(\theta)$ is a well-defined measurable function on $\L_c X$. Since it is also bounded by boundedness of $\theta$, it is moreover integrable, and we define $I: \mathscr{D} \rightarrow \R$ by the formula
\begin{equation} \label{DefinitionI}
  I[\theta] ~ \stackrel{\text{def}}{=} ~ \mathbb{W}\left[ \tilde{q}(\theta) \exp\left(-\frac{1}{8} \int_{\T} \mathrm{scal}\bigl(\gamma(\tau)\bigr) \mathrm{d} \tau\right)\right].
\end{equation}
The main difference of this definition to the formal version \eqref{RewriteI} is the appearance of the exponential including the scalar curvature. While this may seem strange at glance, this is an important ``quantum correction'' to the definition, c.f.\ Remark~\ref{RemarkScalarCurvature} below. 

\medskip

\noindent \textbf{Examples.} We now give some examples of differential forms that are contained in $\mathscr{D}$, together with their $I$-integrals. We assume $X$ to be a compact spin manifold in this section. Given a differential form $\vartheta \in \Omega^\ell(X)$ and $\tau \in \T$, we can produce a differential $\ell$-form $\vartheta(\tau) \in \Omega^\ell(\L X)$ by setting
\begin{equation} \label{TrivialLift}
  \vartheta(\tau)_\gamma[v_1, \dots, v_\ell] \stackrel{\text{def}}{=} \vartheta_{\gamma(\tau)}\bigl[v_1(\tau), \dots, v_\ell(\tau)\bigr]
\end{equation}
for $v_1, \dots, v_\ell \in T_\gamma \L X$. Moreover, for any function $\varphi \in C^\infty(\T)$, we can construct another differential form $\overline{\vartheta} \in \Omega^\ell(\L X)$ by setting
\begin{equation} \label{AdmissibleFormExample}
\overline{\vartheta} = \int_{\T} \vartheta(\tau) \dd \tau \qquad \text{or}\qquad \overline{\vartheta}_\gamma[v_1, \dots, v_\ell] = \int_{\T} \vartheta_{\gamma(\tau)}\bigl[v_1(\tau), \dots, v_\ell(\tau)\bigr] \dd \tau.
\end{equation}
If $\vartheta \in \Omega^1(X)$, then $\overline{\vartheta} \in \Omega^1(\L X)$ as defined in \eqref{AdmissibleFormExample} satisfies the assumptions (i) and (ii) of Notation~\ref{NotationD}, hence sums of wedge products of forms of this type are contained in the domain $\mathscr{D}$.
 
On these forms, the integral map is given as follows: Given $\vartheta_a \in \Omega^{1}(X)$, $a = 1, \dots, N$, and correspondingly $\overline{\vartheta}_a \in \Omega^1(\L X)$ defined by \eqref{AdmissibleFormExample}, then $\overline{\vartheta}_1 \wedge \cdots \wedge \overline{\vartheta}_N \in \mathscr{D}$, and the corresponding integral is given by the combinatoric formula
\begin{equation} \label{IntegralMapCylinderForms}
\begin{aligned}
I\bigl[\overline{\vartheta}_1 \wedge \cdots \wedge \overline{\vartheta}_N\bigr] = 2^{-N/2} \sum_{\sigma \in S_N} \mathrm{sgn}(\sigma) \int_{\Delta_N}  \Str \Bigl(e^{-\tau_1 H}  \prod_{a=1}^N \cc(\vartheta_{\sigma_a}) e^{-(\tau_a-\tau_{a-1})H} \Bigr)\dd \tau
\end{aligned}
\end{equation}
where $H = \mathsf{D}^2/2$, with $\mathsf{D}$ the Dirac operator. This follows from the explicit formula \eqref{FormulaFromThm}; the Wiener integral in \eqref{DefinitionI} is then evaluated using a vector-valued Feynman-Kac formula, see e.g.\ \cite{GueneysuVectorValuedFK}.

\begin{remark} \label{RemarkScalarCurvature}
The scalar curvature factor of \eqref{DefinitionI} is needed because of the Lichnerowicz formula; without it, formula \eqref{IntegralMapCylinderForms} would feature the operator $H - \mathrm{scal}/8$ instead of $H$, which has no good cohomological properties: It turns out that the scalar curvature term is necessary in order to make the functional coclosed (or, in physics lingo: to make the path integral supersymmetric).
\end{remark}



\section{Second Construction: The Chern Character} \label{SectionSecondConstruction}

The above construction of the integral map was achieved by a na\"ive reformulation \eqref{RewriteI} of the heuristic path integral formula \eqref{FormalDefI}. Its disadvantage is that the domain $\mathscr{D} \subset \Omega(\L X)$ where it is defined is quite small;  for example it does {\em not} contain the Bismut-Chern characters considered below, which are the most interesting integrands due to their r\^ole played in relation to the index theorem. We therefore now give a different construction of a path integral map, which has a much larger domain of definition and turns out to extend the previous one. A complete account can be found in the paper \cite{GueneysuLudewig}. 

\medskip

\noindent \textbf{The bar construction.} To set things up, we have to introduce the following algebraic machinery: The {\em bar complex} associated to a differential graded algebra $\Omega$ is the graded vector space\footnote{Here $\Omega(X)[1]$ equals $\Omega(X)$ as a vector space, but with degrees shifted by one: We have $\vartheta \in \Omega^{k+1}$ if and only if $\vartheta \in \Omega[1]^{k}$.}
\begin{equation} \label{BarComplex}
  \mathsf{B}\bigl(\Omega\bigr) = \bigoplus_{N=0}^\infty \Omega[1]^{\otimes N}.
\end{equation}
The elements of $\B(\Omega)$ are called bar chains and denoted by $(\vartheta_1, \dots, \vartheta_N)$ for $\vartheta_a \in \Omega$, suppressing the tensor product sign in notation for convenience. $\B(\Omega)$ has a distinguished subspace $\B^\natural(\Omega)$, which consists of those elements of $\B(\Omega)$ that are invariant under graded cyclic permutation of the tensor factors. $\B(\Omega)$ has two differentials, a differential $d$ coming from the differential of $\Omega$ and the {\em bar differential} $b^\prime$; they are given by
\begin{equation*}
\begin{aligned}
d(\vartheta_1, \dots, \vartheta_N) &= \sum_{k=1}^N (-1)^{n_{k-1}}({\vartheta}_1, \dots, {\vartheta}_{k-1}, d \vartheta_k, \dots, \vartheta_N)\\
b^\prime(\vartheta_1, \dots, \vartheta_N) &= -\sum_{k=1}^{N-1} (-1)^{n_k}({\vartheta}_1, \dots, {\vartheta}_{k-1}, {\vartheta}_{k}\vartheta_{k+1}, \vartheta_{k+2}, \dots, \vartheta_N)
\end{aligned}
\end{equation*}
where $n_k = |\vartheta_1|+\dots + |\vartheta_k|-k$. The above differentials satisfy $db^\prime + b^\prime d = 0$, hence turn $\B(\Omega)$ and $\B^\natural(\Omega)$ into bicomplexes with total differential $d + b^\prime$. Dually, we have the {\em codifferential}
\begin{equation} \label{Codifferential}
  (\delta \ell)[\vartheta_1, \dots, \vartheta_N] \stackrel{\text{def}}{=} - \ell\bigl[(d+b^\prime)(\vartheta_1, \dots, \vartheta_N)\bigr].
\end{equation}
on the space of linear maps $\ell: \B(\Omega) \rightarrow \C$.


\medskip

\noindent \textbf{The iterated integral map.} Remember the definition \eqref{TrivialLift} of cylinder forms above. {\em Chen's iterated integral map} \cite{Chen1} also constructs differential forms on the loop space from differential forms on $X$, this time taking as input elements of the bar complex $\B(\Omega(X))$. For our purposes, we need an extension of this, introduced by Getzler, Jones and Petrack \cite{GJP}. We consider the differential graded algebra
\begin{equation*}
  \Omega_{\T}(X) \stackrel{\text{def}}{=} \Omega(X \times \T)^{\T},
\end{equation*}
the space of differential forms on $X \times \T$ which are constant in the $\T$-direction. Elements $\vartheta \in \Omega_\T(X)$ will always be decomposed into $\vartheta = \vartheta^\prime + dt \wedge \vartheta^{\prime\prime}$, where $\vartheta^\prime, \vartheta^{\prime\prime} \in \Omega(X)$. The differential of $\Omega_\T(X)$ is $d_\T := d- \iota_{\partial_t}$, where $\iota_{\partial_t}$ denotes insertion of the canonical vector field $\partial_t$ on the $\T$ factor and $d$ denotes the de-Rham differential on $X \times \T$. In other words, we have
\begin{equation*}
  d_\T\vartheta = d_\T(\vartheta^\prime + dt \wedge \vartheta^{\prime\prime}) = d \vartheta^\prime - dt \wedge d\vartheta^{\prime\prime} - \vartheta^{\prime\prime},
\end{equation*} 
where now on the right hand side, $d$ denotes the de-Rham differential on $X$.
The version of the {\em extended iterated integral map} used in this survey is a map taking $\B(\Omega_\T(X))$ to $\Omega(\L X)$; it is defined by the formula
\begin{equation} \label{IteratedIntegralMap}
 \rho(\vartheta_1, \dots, \vartheta_N) = \int_{\Delta_N} \bigl(\iota_K \vartheta_1^\prime(\tau_1) + \vartheta_1^{\prime\prime}(\tau_1)\bigr) \wedge \cdots \wedge \bigl(\iota_K \vartheta_N^\prime(\tau_N) + \vartheta_N^{\prime\prime}(\tau_N)\bigr) \dd \tau\end{equation}
for $\vartheta_1, \dots, \vartheta_N \in \Omega_\T(\L X)$, where we recall that $K(\gamma) = \dot{\gamma}$ is the canonical velocity vector field.
This allows to produce many examples of differential forms on the loop space. 

The crucial fact about $\rho$ is that its restriction $\rho^\natural$ to cyclic chains
\begin{equation*}
  \rho^\natural: \B\bigl(\Omega_\T(X)\bigr) \supset \B^\natural \bigl(\Omega_\T(X)\bigr) \longrightarrow \Omega(\L X)^{\T} \subset \Omega(\L X)
\end{equation*}
is a {\em chain map} in the sense that $\rho^\natural$ sends the total differential $d_\T+b^\prime$ to the equivariant differential $d_K$ (defined in \eqref{EquivariantDifferential}). Moreover, notice that the degree shift in the definition of $\B(\Omega_\T(X))$ ensures that $\rho$ is in fact degree-preserving.

\medskip

\noindent \textbf{The Chern character.} Let now $X$ be a compact spin manifold. Our second construction of the path integral map $I$ is based on the construction of a closed cochain
\begin{equation*}
  \mathrm{Ch}_{\mathsf{D}}: \B^\natural\bigl(\Omega_\T(X)\bigr) \longrightarrow \R,
\end{equation*}
called the {\em Chern character} in \cite{GueneysuLudewig}. It has the property that it vanishes on the kernel $\ker(\rho)$ of the iterated integral map \eqref{IteratedIntegralMap}, hence $\mathrm{Ch}_{\mathsf{D}}$ can be pushed forward to a functional on the image of the iterated integral map inside $\Omega(\L X)$.

To define $\mathrm{Ch}_{\mathsf{D}}$, we define a cochain $F$ on $\B(\Omega_\T(X))$ with values in the algebra of linear operators on $L^2(X, \Sigma)$. Explicitly, $F$ is given on homogeneous elements by the formula
\begin{equation*}
\begin{aligned}
  F[\vartheta] &\stackrel{\text{def}}{=} \cc(\vartheta^{\prime\prime}) + [\mathsf{D}, \cc(\vartheta^\prime)]- \cc(d \vartheta^\prime), \\
  F[\vartheta_1, \vartheta_2] &\stackrel{\text{def}}{=} (-1)^{|\vartheta_1^\prime|}\bigl(\cc({\vartheta}_1^\prime)\cc({\vartheta}_2^\prime) - \cc({\vartheta}_1^\prime \wedge {\vartheta}_2^\prime) \bigr),
  \end{aligned}
\end{equation*}
where $\mathsf{D}$ is the Dirac operator;
moreover, we set $F[\vartheta_1, \dots, \vartheta_k] = 0$ whenever $k \geq 3$.
The formula for the Chern character is now
\begin{equation} \label{ChernCharacter}
\begin{aligned}
  \mathrm{Ch}_{\mathsf{D}}[\vartheta_1, \dots, \vartheta_N] = 2^{-n_N/2} \sum_{s \in \mathscr{P}_{N}} \int_{\Delta_M} \mathrm{Str}\Bigl(&e^{-\tau_1 H} \prod_{a=1}^M F[\vartheta_{s_{a-1}+1}, \dots \vartheta_{s_a}] e^{-(\tau_a - \tau_{a-1})H} \Bigr)\dd \tau.
\end{aligned}
\end{equation}
Here $\mathscr{P}_N$ denotes the set of all partitions of $\{1, \dots, N\}$, given by a sequence of numbers $s = \{0 = s_0 < s_1 < \dots < s_M = N\}$. In particular, as $F$ vanishes when one inputs more than two elements, a summand corresponding to a partition $s$ is zero as soon as there exists an index $a$ with $s_a - s_{a-1} \geq 3$.

\begin{remark}
The name Chern character stems from the fact that $\mathrm{Ch}_{\mathsf{D}}$ can be interpreted as the version of a Chern character in non-commutative geometry, namely that of a Fredholm module given by the Dirac operator on $X$. For details, see \cite{GueneysuLudewig}.
\end{remark}

\medskip

\noindent \textbf{Properties of the Chern character.}
One of the advantages of the second approach to the supersymmetric path integral map is that due to the algebraic character of the construction, it is easier to investigate its properties. As mentioned above, one of the results is that $\mathrm{Ch}_{\mathsf{D}}$ is {\em Chen normalized} \cite[Thm.~5.5]{GueneysuLudewig}, meaning that it vanishes on the kernel $\ker(\rho^\natural)$ of the iterated integral map, restricted to cyclic chains. This means that we can define its push-forward
\begin{equation*}
  I^\prime: \Omega(\L X) \supset \mathrm{im}(\rho^\natural) \longrightarrow \R, \qquad I^\prime[\theta] = \mathrm{Ch}_{\mathsf{D}}\bigl[\rho^\natural(\vartheta_1, \dots, \vartheta_N)\bigr]
\end{equation*}
if $\theta = \rho^\natural(\vartheta_1, \dots, \vartheta_N)$; notice that this is well-defined as $\mathrm{Ch}_{\mathsf{D}}$ is Chen normalized. This gives a second functional on the space of differential forms on the loop space, with domain $\mathrm{im}(\rho^\natural)$. One of the main features of the construction is the fact that $\mathrm{Ch}_{\mathsf{D}}$ is coclosed, meaning that
\begin{equation} \label{ChClosedness}
 \delta \mathrm{Ch}_{\mathsf{D}}= 0,
\end{equation}
where $\delta$ is the codifferential \eqref{Codifferential}; c.f.\ \cite[Thms~4.2, 5.3]{GueneysuLudewig}. By the compatibility of $\rho^\natural$ with respect to the differentials, this implies that $I^\prime$ is coclosed with respect to the equivariant differential $d_K$. In other words, for any differential form $\theta \in \mathrm{im}(\rho^\natural)$, we have the following version {\em Stokes' theorem}
\begin{equation*}
  I^\prime[d_K \theta] = 0,
\end{equation*}
stating that exact forms have vanishing integral. In physics slang, this is the {\em supersymmetry} of the path integral.

However, much more is true. The operator-theoretic formula \eqref{ChernCharacter} for $\mathrm{Ch}_{\mathsf{D}}$ makes it accessible to Getzler's rescaling technique; a souped up version of this machinery then enables to show the following \cite[Thm.~9.1]{GueneysuLudewig}.

\begin{theorem} \label{ThmChCohomologous}
$\mathrm{Ch}_{\mathsf{D}}$ is cohomologous, as a Chen normalized cochain on $\B^\natural(\Omega_\T(X))$, to the Chen normalized cochain $\mu_0$, defined by
\begin{equation} \label{DefinitionCh0}
  \mu_0(X)[\vartheta_1, \dots, \vartheta_N] \stackrel{\mathrm{def}}{=} \frac{1}{(2 \pi)^{n/2} N!}\int_X \hat{A}(X) \wedge \vartheta_1^{\prime\prime} \wedge \cdots \wedge \vartheta_N^{\prime\prime},
\end{equation}
where $\hat{A}(X)$ is the Chern-Weil representative of the $\hat{A}$-genus of $X$.
\end{theorem}

Remember here that we say that a cochain is Chen normalized if it vanishes on the kernel of $\rho$. Since $\hat{A}(X)$ is a closed differential form, this implies \eqref{ChClosedness}, by the usual Stokes theorem. As we discuss below, \eqref{DefinitionCh0} essentially implies the localization formula \eqref{LocalizationFormula} for suitable differential forms.

\medskip

\noindent \textbf{Comparison to the previous definition.}
Inspecting formula \eqref{IteratedIntegralMap}, we see that
\begin{equation*}
  \rho(\vartheta_1, \dots, \vartheta_N) = \int_{\Delta_N} \vartheta^{\prime\prime}_1(\tau_1) \wedge \cdots \wedge \vartheta_N^{\prime\prime}(\tau_N) \dd \tau,
\end{equation*}
whenever $\vartheta_a^\prime = 0$ for each $a$. If each $\vartheta^{\prime\prime}_a$ has degree one, we have
\begin{equation*}
  \sum_{\sigma \in S_N} \sgn(\sigma) \rho(\vartheta_{\sigma_1}, \dots, \vartheta_{\sigma_N}) = \overline{\vartheta}^{\prime\prime}_1 \wedge \cdots \wedge \overline{\vartheta}^{\prime\prime}_N,
\end{equation*}
 which is contained in $\mathscr{D}$, hence has a well-defined path integral, as defined in \eqref{DefinitionI}. On the other hand, an inspection of the formula \eqref{ChernCharacter} yields
\begin{equation*}
\begin{aligned}
\mathrm{Ch}_{\mathsf{D}}[\vartheta_1, \dots, \vartheta_N] &= 2^{-N/2} \int_{\Delta_N} \mathrm{Str}\Bigl(e^{-\tau_1 H} \prod_{a=1}^N\cc(\vartheta_a^{\prime\prime}) e^{-(\tau_a - \tau_{a-1})H} \Bigr)\dd \tau.
  \end{aligned}
\end{equation*}
After anti-symmetrization, this coincides with the formula for $I[\overline{\vartheta}^{\prime\prime}_1 \wedge \cdots \wedge \overline{\vartheta}^{\prime\prime}_N]$, as calculated in  \eqref{IntegralMapCylinderForms}. 
In this sense, the two versions of the integral map agree, and we from now on, we will use the notation $I$ instead of $I^\prime$.

\section{Bismut-Chern Characters, Entire Chains and the Localization Formula}

As usual, throughout this section $X$ denotes a Riemannian manifold, which is assumed to be compact and spin for all statements related to the path integral map.

\medskip

\noindent \textbf{Periodic cyclic cohomology.}
Throughout, a general differential form on the loop space is the direct {\em sum} of its homogeneous components, in other words, we denote
\begin{equation*}
  \Omega(\L X) \stackrel{\mathrm{def}}{=} \bigoplus_{\ell=0}^\infty \Omega^\ell(\L X).
\end{equation*}
It is well-known however \cite{MR1010411}, that in the equivariant cohomology of the loop space, it is important to allow differential forms that are an infinite sum of its homogeneous components, in other words, elements of the direct {\em product} of the $\Omega^\ell(\L X)$. This gives the {\em periodic equivariant cohomology} $h_\T(\L X)$ of the loop space, which is the cohomology of the $\Z_2$-graded complex\footnote{It customary in this context to introduce a formal variable of degree $2$ and its inverse in order to define the periodic cyclic cohomology. The effect is that the complex and its cohomology are $\Z$-graded, but $2$-periodic. Here we reduce modulo $2$ right away.} $\widehat{\Omega}(\L X)^\T = \widehat{\Omega}^+(\L X)^\T \oplus \widehat{\Omega}^-(\L X)^\T$, where
\begin{equation*}
  \widehat{\Omega}^{+}(\L X) \stackrel{\mathrm{def}}{=} \prod_{\ell=0}^\infty \Omega^{2\ell}(\L X), \qquad \widehat{\Omega}^{-}(\L X) \stackrel{\mathrm{def}}{=} \prod_{\ell=0}^\infty \Omega^{2\ell+1}(\L X).
\end{equation*}
The corresponding differential is the equivariant differential $d_K$, c.f.\ \eqref{EquivariantDifferential}, which exchanges the even and the odd part.

\medskip

\noindent \textbf{Bismut-Chern-characters.} Maybe the most prominent example of such a differential form are the Bismut-Chern-characters, defined as follows.

\begin{definition}
Let $E$ be a Hermitean vector bundle with connection $\nabla$ over the manifold $X$. The {\em Bismut-Chern-character} associated to this data is the equivariantly closed differential form $\mathrm{Ch}(E, \nabla) \in \widehat{\Omega}^+(\L X)$ given by the formula
\begin{equation*}
  \mathrm{Ch}(E, \nabla)_\gamma = \sum_{N=0}^\infty (-1)^N \int_{\Delta_N} \tr_E\left( [\gamma\|_{\tau_N}^1]^E \prod_{a=1}^N R(\tau_a)_\gamma [\gamma\|_{\tau_{a-1}}^{\tau_a}]^E\right) \dd \tau
\end{equation*}
at $\gamma \in \L X$, where $R$ is the curvature of the connection $\nabla$. 
\end{definition}

Explicitly, the degree $2N$-component $\mathrm{Ch}_N$ of $\mathrm{Ch}(E, \nabla)$ is given by
\begin{equation*}
  \mathrm{Ch}_N[v_{2N}, \dots, v_{1}] = 2^{-N}\sum_{\sigma \in S_{2N}}\int_{\Delta_N} \tr_E \left([\gamma\|_{\tau_N}^1]^E \prod_{a=1}^N R\bigl(v_{\sigma_{2a}}(\tau_a), v_{\sigma_{2a-1}}(\tau_a)\bigr) [\gamma\|_{\tau_{a-1}}^{\tau_a}]^E\right) \dd \tau.
\end{equation*}
The main properties of the Bismut-Chern-character is that it is equivariantly closed, $d_K \mathrm{Ch}(E, \nabla) = 0$, in other words, $d\mathrm{Ch}_N = \iota_K \mathrm{Ch}_{N+1}$, and that its pullback along the inclusion $i: X \rightarrow \L X$ is the ordinary Chern character of $(E, \nabla)$, defined using Chern Weyl-theory:
\begin{equation} \label{PullbackBismutChern}
 i^* \mathrm{Ch}(E, \nabla) =  \ch(E, \nabla).
\end{equation} 
Formally, the following theorem has been observed by Bismut \cite{MR786574} and was his original motivation for the definition of these differential forms.
Of course, by the usual argument of McKean-Singer \cite[Thm.~3.50]{MR2273508}, the right hand side of \eqref{PathIntegralBismutChern} below equals $\mathrm{ind}(\mathsf{D}_E)$, the graded index of the twisted Dirac operator $\mathsf{D}_E$. 

\begin{theorem} \label{TheoremBismutChern}
  We have the formula
 \begin{equation} \label{PathIntegralBismutChern}
   I\bigl[\mathrm{Ch}(E, \nabla)\bigr] = \Str(e^{-\mathsf{D}_E^2/2}).
 \end{equation}
\end{theorem}

In \eqref{PathIntegralBismutChern}, we take $I$ to be the path integral map constructed in Section~\ref{SectionSecondConstruction}.
This makes sense as by the results of \cite[\S 6]{GJP}, $\mathrm{Ch}(E, \nabla)$ can be written as an iterated integral, i.e.\ there exists elements $c_N \in \B^{2N}(\Omega_\T(X))$, $N=0, 1, 2, \dots$ such that $\rho(c_N) = \mathrm{Ch}_N$. In particular, each $\mathrm{Ch}_N$ is contained in the domain of the integral map $\rho_*^\natural\mathrm{Ch}_{\mathsf{D}}$ constructed in Section~\ref{SectionSecondConstruction}. The identity \eqref{PathIntegralBismutChern} is then proven using Prop.~8.2 of \cite{GueneysuLudewig}. 

We remark that $\mathrm{Ch}$ does {\em not} directly lie in the domain of the integral map $I$ defined in Section~\ref{SectionFirstConstruction}; in fact, $\mathrm{Ch}$ is not even a smooth differential form on $\L_c X$, due to the presence of the parallel transport in its definition. However, interpreting the parallel transport in the stochastic sense, one obtains a differential form on $\L_c X$ with measurable coefficients. The top degree map can be applied to this measurable differential form, which yields a measurable function on $\L_c X$. One can then compute $I[\mathrm{Ch}(E, \nabla)]$ by employing a suitable version of the Feynman-Kac-formula, which gives the same result.

\medskip

\noindent \textbf{Entire cohomology.} In the discussion of Thm.~\ref{TheoremBismutChern}, we have so far omitted the fact that the Bismut-Chern-characters are not contained in $\Omega(\L X)$, but only in the extension $\widehat{\Omega}(\L X)$ that allows infinite sums of homogeneous forms. In particular, it is not at all clear {\em a priori} that $I[\mathrm{Ch}(E, \nabla)]$, defined as the sum of the individual integrals $I[\mathrm{Ch}_N]$ makes any sense. This issue is best discussed in our second approach to the integral map, where it is related to the entire cohomology of Connes.

For a differential graded algebra $\Omega$, we denote by $\widehat{\B}(\Omega)$ the complex defined by the same formula \eqref{BarComplex} as $\B(\Omega)$, but with a direct product replacing the direct sum. In other words, its elements are arbitrary sums $\sum_{N=0}^\infty \theta^{(N)}$, with $\theta^{(N)} \in \Omega[1]^{\otimes N}$, without any convergence requirement. The {\em entire bar complex} $\B_\epsilon(\Omega)$ is then a certain subcomplex of $\widehat{\B}(\Omega)$, containing chains that satisfy a certain growth condition; for details, we refer to \cite{GueneysuLudewig}. One can then show that for any Bismut-Chern-character $\mathrm{Ch}(E, \nabla)$, the chain $c = \sum_{N=0}^\infty c_N \in \widehat{B}(\Omega_\T(X))$ such that $\mathrm{Ch}(E, \nabla) = \rho(c)$, constructed by Getzler-Jones-Petrack, is entire. Dually, the following result is shown in \cite[Thms~4.1, 5.2]{GueneysuLudewig}:

\begin{theorem}
The Chern character $\mathrm{Ch}_{\mathsf{D}}$ has a continuous extension to $\B_\epsilon(\Omega)$.
\end{theorem}

Together with the discussion before, this gives an a priori reason why the left hand side of \eqref{PathIntegralBismutChern} is well-defined.

\medskip

\noindent \textbf{The localization formula and the index theorem.} 
We now explain how to rigorously conduct the proof of the Atiyah-Singer index theorem envisioned by Atiyah \cite{MR816738} and Bismut \cite{MR786574} using our results. The first result is the following localization formula. We say that $\theta \in \widehat{\Omega}(\L X)$ is an entire iterated integral, if there exists $c \in \B_\epsilon^\natural(\Omega_\T(X))$ such that $\theta = \rho^\natural(c)$.

\begin{theorem}
Let $\theta \in \widehat{\Omega}(\L X)$ be equivariantly closed, i.e.\ $d_K \theta = 0$, and assume that it is an entire iterated integral. Then
\begin{equation} \label{LocalizationFormula2}
  I[\theta] = (2\pi)^{-n/2}\int_X \hat{A}(X) \wedge i^*\theta.
\end{equation}
\end{theorem}

\begin{proof}
By the assumption on $\theta$, there exists $c \in \B_\epsilon^\natural(\Omega_\T(X))$ with $\rho^\natural(c) = \theta$. 
Define $I_0 : \widehat{\Omega}(\L X) \rightarrow \R$ by setting $I_0[\theta]$ to be the right hand side of \eqref{LocalizationFormula2} and notice that by the definition \eqref{IteratedIntegralMap} of the iterated integral map, we have
\begin{equation*}
  I_0[\theta] = \mu_0(X)[c],
\end{equation*}
where  $\mu_0(X)$ is defined in \eqref{DefinitionCh0}.
Now by Thm.~\ref{ThmChCohomologous}, there exists a Chen normalized cochain $\mu^\prime$ such that $\mathrm{Ch}_{\mathsf{D}} - \mu_0(X) = \delta \mu^\prime$. Therefore,
\begin{equation*}
\begin{aligned}
  I[\theta] - I_0[\theta] = (I-I_0)[\rho^\natural(c)] &=\bigl(\mathrm{Ch}_{\mathsf{D}} - \mu_0(X)\bigr)[c] \\
  &=\delta \mu^\prime [c] = - \mu^\prime\bigl[(d_\T + b^\prime)c\bigr].
\end{aligned}
\end{equation*}
Since $\rho^\natural$ is a chain map and $\theta$ is equivariantly closed, the calculation
\begin{equation*}
  \rho^\natural\bigl((d_\T+b^\prime)c\bigr) = d_K \rho^\natural(c) = d_K \theta = 0
\end{equation*}
shows that $(d_\T+b^\prime)c \in \ker(\rho^\natural)$, hence $\mu[(d_\T + b^\prime)c]= 0$, as $\mu$ is Chen normalized.
\end{proof}

The localization formula \eqref{LocalizationFormula2} is an infinite-dimensional version of the localization formula of equivariant cohomology in finite dimensions, c.f.\ \cite{BerlineVergne, DuistermaatHeckmann}. Applying it to a Bismut-Chern character (which is both equivariantly closed and can be represented as an entire iterated integral, as discussed above), we get
\begin{equation} \label{OtherEvaluation}
  I\bigl[\mathrm{Ch}(E, \nabla)\bigr] = (2\pi)^{-n/2}\int_X \hat{A}(X) \wedge i^*\mathrm{Ch}(E, \nabla) = (2\pi)^{-n/2} \int_X \hat{A}(X) \wedge \mathrm{ch}(E, \nabla),
\end{equation}
where in the last step, we used \eqref{PullbackBismutChern}. Together with our previous formula \eqref{PathIntegralBismutChern} and the McKean-Singer formula, this proves the Atiyah-Singer index theorem.

\medskip

\noindent \textbf{Odd dimensions.} We remark that nowhere in the above, it was necessary to restrict to even-dimensional manifolds. Of course, in odd dimensions, both \eqref{OtherEvaluation} and \eqref{PathIntegralBismutChern} are zero; for $I[\mathrm{Ch}(E, \nabla)]$, this is true because the path integral is an odd functional in this case, in other words, it evaluates as zero on even-dimensional forms. To obtain a non-trivial result in this case, one uses the {\em odd Bismut-Chern character} $\mathrm{Ch}(g)$ of Wilson \cite{Wilson}, an equivariantly closed, odd element of $\widehat{\Omega}(\L X)$ associated to a map $g: X \rightarrow \mathrm{U}(k)$, the unitary group of order $k$, for some $k$. This can be represented by an entire iterated integral following the work of Cacciatori-G\"uneysu \cite{CacciatoriGueneysu}. A result similar to Thm.~\ref{TheoremBismutChern} connects this to the spectral flow to the family $\mathsf{D}_s = \mathsf{D} + s \cc(g^{-1}dg)$ of Dirac operators on $\Sigma \otimes \C^k$. This recovers the odd index theorem of Getzler \cite{MR1231957}.

\bibliography{LiteraturMatrixReport}

\end{document}